\documentclass[11pt, oneside]{amsart}   	
\usepackage{geometry}                		
\usepackage{graphicx}				
\usepackage{amssymb}
\usepackage{tikz}
\usepackage{hyperref}

\newtheorem{theorem}{Theorem}[section]
\newtheorem{lemma}[theorem]{Lemma}
\newtheorem{cor}[theorem]{Corollary}

\theoremstyle{definition}
\newtheorem{definition}[theorem]{Definition}
\newtheorem{example}[theorem]{Example}

\theoremstyle{remark}
\newtheorem{remark}[theorem]{Remark}

\numberwithin{equation}{section}

\begin{document}


\title{Sharp Bounds for Generalized Zagreb Indices of Graphs}
\author{Sanju Vaidya}
\address{Math \& CS Department, Mercy University, 555 Broadway, Dobbs Ferry, NY 10522}
\curraddr{}
\email{SVaidya@mercy.edu}
\thanks{}

\author{Jeff Chang}
\address{Math \& CS Department, Mercy University, 555 Broadway, Dobbs Ferry, NY 10522}
\curraddr{}
\email{cchang4@mercy.edu}
\thanks{}

\subjclass[2010]{Primary }


\begin{abstract}

In the last forty years, many scientists used graph theory to develop mathematical models for analyzing structures and properties of various chemical compounds.  In this paper, we will establish formulas and bounds for generalized first Zagreb Index and coindex, which are based on degrees of vertices.  In addition, for triangle and quadrangle free graphs, we will establish formulas and bounds for generalized first leap Zagreb Index and coindex, which are based on 2-distance degrees of vertices. Additionally, we will establish sharp bounds of generalized first Zagreb index and the leap index for various types of graphs and provide examples for which the sharp bounds are attained. In addition, we will find regression models and compare the first Zagreb index and the first leap Zagreb index for predicting some physicochemical properties of certain chemical compounds, benzenoid hydrocarbons.

\end{abstract}	

\maketitle

\section{Introduction}

In mathematical chemistry, an emerging interdisciplinary research area, chemical graph theory brings together mathematicians, computer scientists, and chemists to analyze complicated problems such as design of new drugs and development of new agrochemicals. In the last forty years, many scientists have developed mathematical models for analyzing structures and properties of various chemical compounds. Graph theory plays a very important role in developing many types of models such as Quantitative Structure-Property Relationships (QSPR) models, Quantitative Structure-Activity Relationships (QSAR) models, and Quantitative Structure-Toxicity Relationships (QSTR) models (\cite{BGG}, \cite{DB}, \cite{GFS}, \cite{SPNA}, \cite{TC}).

In molecular graphs of chemical compounds vertices correspond to atoms and edges correspond to the bonds between them. A topological index (connectivity index) is a type of a molecular descriptor that is based on the molecular graph of a chemical compound. In 1947 Harry Weiner introduced a topological index related to molecular branching. He correlated the indices with the boiling points of certain chemical compounds, alkanes. This inspired many mathematicians and chemists to develop more topological indices for molecular graphs.

In 1972, Zagreb indices were introduced by Gutman and Trijanastic \cite{GT}. They are based on the degrees of the vertices. They are very useful in modeling chemical and biological properties of chemical compounds, as shown by Devillers et al \cite{DB}, Basak et al \cite{BGG}, and Todeschin et al \cite{TC}. Gutman et al \cite{GFMG} and Zhou et al \cite{B} used Zagreb indices for analyzing bounds for Estrada index, which is based on the eigenvalues of the adjacency matrix. Zagreb indices were generalized by many mathematicians and scientists. The first Zagreb index of a graph is the sum of the squares of the degrees of the vertices.  Li and Zhao \cite{LZ}, Li and Zheng \cite{LZ2} introduced the general first Zagreb index (zeroth-order general Randic index), which is the sum of any powers of the degrees of the vertices. In 2015, Furtula and Gutman \cite{FG} named the sum of the cubes of the degrees of the vertices as the forgotten topological index or F-index and investigated its properties and applications to Chemistry. Moreover, in 2017, Abdo, Dimitrov, and Gutman \cite{ADG} determined trees with maximal F-index. In 2008, Doslic \cite{D} introduced the first Zagreb coindex, which is based on the degrees of nonadjacent vertices. Additionally, in 2012 Mansour-Song \cite{MS} introduced the general first Zagreb coindex and proved relation between the general first Zagreb index and coindex. Moreover, many mathematicians (such as Gao-Farahani-Li \cite{GFS} and Solemani-Nikmehr-Tavvallaee \cite{SNT}) calculated Zagreb index and coindex for various nanostructures, which are very useful in designing new medicines. 

In 2017 Naji et al \cite{NNI} introduced the  first leap Zagreb index, which is based on 2-distance degree of vertices. In 2019 Kulli V. R.  \cite{K} introduced the general first leap Zagreb index, which is also based on 2-distance degree of vertices. In 2022, Ferdose and Shivashankara \cite{FK} introduced the  first leap Zagreb coindex. In 2022, P.Sarkar, Nilanjan De, and Anita Pal \cite{SPNA} investigated predictive ability of physicochemical prperties of polycyclic aromatic hydrocarbons for cerain degree based topological indices including Zagreb indices. In addition, in 2023, K.S. Pise and H.S. Ramane \cite{PKH} showed chemical importance of the leap Zagreb indices.

The main research question is how to compute bounds and formulas of the indices stated above and find graphs which attain the bounds. Gutman et al \cite{GFMG} established a formula of the first Zagreb index for alkanes (chemical trees). Vaidya-Surendran \cite{VS} extended this to cycloalkanes, alkenes, and alkynes. We established formulas of the first Zagreb index for certain graphs and we proved that the molecular graphs of cycloalkanes, alkenes, and alkynes attain the greatest values. Additionally, Hu-Li-Shi-Xu-Gutman \cite{HLSXG} investigated the zeroth-order general Randic index (which is the same as general Zagreb index) for molecular graphs in which the maximum vertex degree is at most $4$. In this paper we will extend this further to the general first Zagreb index and the leap index for graphs with maximum vertex degree more than $4$.

In Section 2 of this paper, we will review definitions of Zagreb indices and coindices. In Section 3 we will establish formulas of the general first Zagreb index and coindex for graphs. Additionally, we will find upper and lower bounds for the general first Zagreb index and determine graphs which attain the bounds. Moreover, we will establish sharp upper and lower bounds for the general first Zagreb index for some special graphs. In Section 4,  for triangle and quadrangle free graphs, we will establish formulas and bounds for generalized first leap Zagreb Index and coindex, which are based on 2-distance degrees of vertices. In Section 5, we will find regression models for properties like boiling point and entropy of Benzenoid hydrocarbons using these indices. In Section 6 we will have discussion and conclusion. The area of Chemical Graph Theory is fascinating. Various topological indices play a crucial role in modeling many properties of chemical compounds. This process of mathematical modeling is very useful in real life because it reduces cost and saves time in many fields such as drug discovery and assessment of chemicals.

\section{Review of Zagreb indices and coindices}

In this section we will review some results about Zagreb indices and their generalizations. We will use the notation and terminology introduced in Gutman et al \cite{GFMG}. In {\bf molecular graphs} of chemical compounds, vertices correspond to individual atoms and edges correspond to the bonds between them. A {\bf topological index} (connectivity index) is a type of a molecular descriptor that is based on the molecular graph of a chemical compound. The definition of the first Zagreb index, which is introduced by Gutman and Trijanastic \cite{GT}, is as follows.
\begin{definition}
    Let $G$ be a graph with vertex set $V(G)$ and $d(u)$ be the degree of the vertex $u\in V(G)$. Then the first Zagreb index is defined as follows:
    \[Z_G=M_1(G)=\sum_{u\in V(G)}d(u)^2.\]
\end{definition}

In 2007 Gutman-Fortula-Markovic-Gilistic \cite{GFMG} proved that for a chemical tree (alkane) with $n$ vertices $Z_G=M_1(G)=6n-10-2n_2-2n_3$, where $n_i$ is the number of vertices of degree $i$ for i = 2, 3. In Vaidya and Surendran \cite{VS}, we extended this and established formulas of the first Zagreb index for molecular graphs of cycloalkanes, alkenes, and alkynes. 

Zagreb indices have been generalized and studied for more than 30 years. The following general first Zagreb index was first studied by Li and Zhao \cite{LZ}, Li and Zheng \cite{LZ2}, and later on by many mathematicians such as Liu-Liu \cite{LL}.

\begin{definition}
    Let $G$ be a graph with vertex set $V(G)$ and $d(u)$ be the degree of the vertex $u\in V(G)$. Let $\alpha$ be a nonzero real number such that $\alpha\neq 1$. Then the general first Zagreb index (or the zeroth-order general Randic index) is defined as follows:
    \[M_1^\alpha (G)=\sum_{u\in V(G)}d(u)^\alpha.\]
\end{definition}

In 2008, Doslic \cite{D} introduced the following first Zagreb coindex.

\begin{definition}
    Let $G$ be a graph with vertex set $V(G)$, edge set $E(G),$ and $d(u)$ be the degree of the vertex $u\in V(G)$. Then the first Zagreb coindex is defined as follows:
    \[M_1^-(G)=\sum_{\substack{uv\not\in E(G)\\u\neq v}}[d(u)+d(v)].\]
\end{definition}

In 2012 Mansour-Song \cite{MS} introduced the following general first Zagreb coindex.

\begin{definition}
    Let $G$ be a graph with vertex set $V(G)$, edge set $E(G)$, and $d(u)$ be the degree of the vertex $u\in V(G)$. Let $\alpha$ be an integer such that $\alpha\geq 1$. Then the general first Zagreb coindex is defined as follows:
    \[M_1^{-\alpha}(G)=\sum_{\substack{uv\not\in E(G)\\u\neq v}}[d(u)^\alpha+d(v)^\alpha].
    \]
\end{definition}

In 2017 Naji et al \cite{NNI} introduced the following first leap Zagreb index, which is based on 2-distance degrees of vertices.

\begin{definition}
    Let $G$ be a graph with vertex set $V(G)$ and $d_2(u)$ be the 2-distance degree of the vertex $u\in V(G)$, which is the number of vertices which are at distance 2 from the vertex $u$. Then the first leap Zagreb index is defined as follows:
    \[LM_1(G)=\sum_{u\in V(G)}d_2(u)^2.\]
\end{definition}

In 2019 Kulli V. R.  \cite{K} introduced the following general first leap Zagreb index, which is based on 2-distance degrees of vertices.

\begin{definition}
    Let $G$ be a graph with vertex set $V(G)$ and $d_2(u)$ be the 2-distance degree of the vertex $u\in V(G)$. Let $\alpha$ be a nonzero real number such that $\alpha\neq 1$.Then the general first leap Zagreb index is defined as follows:
    \[LM_1^\alpha (G)=\sum_{u\in V(G)}d_2(u)^\alpha.\]
\end{definition}

In 2022 Ferdose and Shivashankara \cite{FK} introduced the following first leap Zagreb coindex.

\begin{definition}
    Let $G$ be a graph with vertex set $V(G)$, edge set $E(G),$ and $d_2(u)$ be the 2-distance degree of the vertex $u\in V(G)$. Let $E_2$ be the edge set of $G^2$, where $G^2$ is the graph with $V(G^2) = V(G)$ and edge $uv \in E_2 = E(G^2)$ if and only if  distance $d(u, v) = 2$ in the graph $G$. Then the first leap Zagreb coindex is defined as follows:
    \[LM_1^-(G)=\sum_{\substack{uv\not\in E_2(G)\\u\neq v}}[d_2(u)+d_2(v)].\]
\end{definition}

\section{Formulas and bounds for generalized Zagreb indices}

In this section we will establish some formulas for the general first Zagreb index and coindex. We will then find sharp upper and lower bounds for the general first Zagreb index using these formulas. First, we need the following Lemma.

\begin{lemma}\label{Lemma}
Let $p$ and $q$ be any positive integers such that $p < q$. Let $\alpha\in\mathbb{R}$ and $\alpha\not\in\{0,1\}$. Then we have the following.
\begin{enumerate}
\item If  the real number $\alpha<0$ or $\alpha>1$, then  $(p + i)^\alpha - p^\alpha - i(\frac{q^\alpha - p^\alpha}{q - p}) \leq 0$ for $1 \leq i \leq q - p - 1$  and if $0 < \alpha < 1$, then $(p + i)^\alpha - p^\alpha - i(\frac{q^\alpha - p^\alpha}{q - p}) \geq 0$ for $1 \leq i \leq q - p - 1$. 

\item If  the real number $\alpha<0$ or $\alpha>1$, then $(p + i)^\alpha - p^\alpha  - i((p + 1)^\alpha - p^\alpha) \geq 0$ for $2 \leq i \leq q - p$ and for $0 < \alpha < 1$, $(p + i)^\alpha - p^\alpha  - i((p + 1)^\alpha - p^\alpha) \leq 0$ for $2 \leq i \leq q - p$.  

\end{enumerate}
\end{lemma}
\begin{proof}
We note that the function $f (x) = \frac {x^\alpha - 1}{x - 1}$, where $x > 1$, is increasing if $\alpha > 1$ or $\alpha < 0$ and is decreasing if $0 < \alpha < 1$. So it follows that if  the real number $\alpha<0$ or $\alpha>1$ and $1 \leq i \leq q - p - 1$,then $f (\frac{p + i}{p}) \leq f(\frac{q}{p})$. By simplifying, we get $(p + i)^\alpha - p^\alpha - i(\frac{q^\alpha - p^\alpha}{q - p}) \leq 0$. In addition, if $0 < \alpha < 1$ and $1 \leq i \leq q - p - 1$, then $f (\frac{p + i}{p}) \geq f(\frac{q}{p})$. By simplifying, we get $(p + i)^\alpha - p^\alpha - i(\frac{q^\alpha - p^\alpha}{q - p}) \geq 0$. The proof is similar for part (2).    
\end{proof}

 In \cite{GD}, Gutman and Das proved the formulas for the first Zagreb index when $\alpha=2$. By generalizing their ideas, now we will establish some formulas for the general first Zagreb index and coindex in the following Theorem.

\begin{theorem}\label{UBFormulas}
    Let $G$ be a graph with $n$ vertices and $m$ edges where $n\geq 3$. Let $n_i$ denote the number of vertices of degree $i$. Let $\delta$ and  $\Delta$ respectively denote the minimum and the maximum degree.  Assume that $\delta \not= \Delta$. Let $\alpha\in\mathbb{R}$ and $\alpha\not\in\{0,1\}$.  Further denote $s_\alpha=\frac{\Delta^\alpha- \delta^\alpha}{\Delta- \delta}$. Then we have the following.
    \begin{enumerate}
        \item The general first Zagreb index
        \[M_1^{\alpha}(G)=n\delta^\alpha +  (2m - n\delta)s_\alpha +\sum_{i =1}^{\Delta - \delta - 1}n_{\delta + i} \left[(\delta + i)^\alpha - \delta^\alpha  - is_\alpha\right].\]
Additionally, if  $\alpha<0$ or $\alpha>1$ the coefficients  $(\delta + i)^\alpha - \delta^\alpha - is_\alpha \leq 0$ for $1 \leq i \leq \Delta -\delta - 1$  and for $0 < \alpha < 1$, the coefficients $(\delta + i)^\alpha - \delta^\alpha - is_\alpha \geq 0$ for $1 \leq i \leq \Delta -\delta - 1$.

        \item The general first Zagreb index
       \begin{align*}
           M_1^{\alpha}(G)=n\delta^\alpha +  (2m - n\delta) (&(\delta + 1)^\alpha - \delta^\alpha)  +\\&\sum_{i = 2}^{\Delta - \delta }n_{\delta + i} \left[(\delta + i)^\alpha - \delta^\alpha  - i((\delta + 1)^\alpha - \delta^\alpha)\right].
       \end{align*}
 Additionally, if  $\alpha<0$ or $\alpha>1$, the coefficients $(\delta + i)^\alpha - \delta^\alpha  - i((\delta + 1)^\alpha - \delta^\alpha) \geq 0$ for $2 \leq i \leq \Delta -\delta$ and for $0 < \alpha < 1$, the coefficients $(\delta + i)^\alpha - \delta^\alpha  - i((\delta + 1)^\alpha - \delta^\alpha) \leq 0$ for $2 \leq i \leq \Delta -\delta $.
 
\item If $\alpha > 1$, then the general first Zagreb coindex
        \begin{align*}
            M_1^{-\alpha+1}&(G) = n(n-1)\delta^\alpha + (2m-n\delta)s_\alpha(n-1)\\
            &+(n-1)\sum_{i=1}^{\Delta - \delta - 1}n_{\delta + i}[(\delta + i)^\alpha - \delta^\alpha  - is_\alpha] - n\delta^{\alpha + 1} - (2m - n\delta)s_{\alpha+1}\\
            &-\sum_{i=1}^{\Delta - \delta -1}n_{\delta + i}[(\delta + i)^{\alpha + 1} - \delta^{\alpha + 1}  - is_{\alpha + 1}].
        \end{align*}

        \item If $\alpha > 1$, then the general first Zagreb coindex
        \begin{align*}
        M_1^{-\alpha+1}&(G) = (n - 1) n\delta^\alpha +  (n - 1)(2m - n\delta) ((\delta + 1)^\alpha - \delta^\alpha) \\
        & + (n - 1) \sum_{i = 2}^{\Delta - \delta }n_{\delta + i} [(\delta + i)^\alpha - \delta^\alpha  - i((\delta + 1)^\alpha - \delta^\alpha)-n\delta^{\alpha + 1})\\ 
         &- (2m - n\delta) ((\delta + 1)^{\alpha + 1} - \delta^{\alpha + 1})\\
         &-\sum_{i = 2}^{\Delta - \delta }n_{\delta + i} \left[(\delta + i)^{\alpha + 1} - \delta^{\alpha + 1}  - i((\delta + 1)^{\alpha + 1} - \delta^{\alpha + 1})\right].
          \end{align*}
    \end{enumerate}
\end{theorem}
\begin{proof}
We have $n=\sum_{i= \delta}^\Delta n_i$ and $\sum_{i= \delta}^\Delta in_i= 2m$. Solving these equations, we will rewrite $n_\Delta$ as follows.
\[
n_\Delta=\frac{2m - n\delta - \sum_{i=\delta + 1}^{\Delta  - 1}((i - d)n_i}{\Delta - \delta}.
\]

Write $M_1^\alpha (G)$ in terms of $n_i$, $\delta\leq i\leq \Delta$, and substitute $n_\delta$ with $n-\sum_{i=\delta + 1}^\Delta n_i$, we obtain
\[
M_1^\alpha (G)=\sum_{i=\delta}^\Delta  i^\alpha n_i = \delta^\alpha( n - \sum_{i= \delta + 1}^\Delta n_i) + \sum_{i=\delta + 1}^\Delta  i^\alpha n_i = n\delta^\alpha + \sum_{i = \delta + 1}^\Delta  n_i(i^\alpha - \delta^\alpha).
\]
Further substitute $n_\Delta$ into the sum and using Lemma \ref{Lemma}, we get the result in part (1). Using the identities $n_\delta+n_{\delta + 1} = n - \sum_{i = \delta + 2}^\Delta n_i$ and $\delta n_\delta + (\delta + 1)n_{\delta + 1} = 2m - \sum_{i = \delta + 2}^\Delta  in_i$, we solve for $n_\delta$ and $n_{\delta + 1}$. Substituting $n_\delta$ and $n_{\delta + 1}$ into $M_1^\alpha(G)=\sum_{i=\delta}^\Delta  i^\alpha n_i$, and using Lemma \ref{Lemma}, we get the result in part (2). To prove parts (3) and (4), we note that Mansour T and Song C. \cite{MS} proved the following relation between the general first Zagreb index and coindex for any graph $G$,
\begin{equation}\label{eq}
    M_1^{-\alpha+1}=(n-1)M_1^\alpha(G)-M_1^{\alpha+1}(G).
\end{equation}
Then using the formula in part (1), we get the result in part (3) and using part (2), we get the result in part (4).

\end{proof}

Solemani-Nikmehr-Tavvallaee \cite{SNT} calculated  the first Zagreb index and coindex for a nanostructure, linear $[n]$ Tetracene, using the first Zagreb polynomial and partition of edges. In the following corollary, we will get the same result by using Theorem \ref{UBFormulas}.

\begin{cor}[Solemani-Nikmehr-Tavvallaee]
    The first Zagreb index and coindex for a linear $[n]$ Tetracene, $G$, are as follows.
    \[
    M_1(G)=122n-20,
    \]
    \[
M_1^-(G)=828n^2-240n+24.
    \]
\end{cor}
\begin{proof}
    We note that for a linear $[n]$ Tetracene, the number of vertices are $18n$ and the number of edges are $23n - 2$. The maximum degree is $3$, the number of vertices of degree $3$ are $10n - 4$, and the number of vertices of degree $2$ are $8n + 4$. Consequently, the results follow from the Theorem \ref{UBFormulas}.
\end{proof} 

The following corollary gives upper bounds or lower bounds for the general first Zagreb index, depending on the values of $\alpha$. Both the upper bounds and lower bounds are sharp. 

\begin{cor}\label{Bounds_Cor}
    Let $G$ be a graph with $n$ vertices and $m$ edges where $n\geq 3$. Let $n_i$ denote the number of vertices of degree $i$. Let $\delta$ and  $\Delta$ respectively denote the minimum and the maximum degree.  Assume that $\delta \not= \Delta$. Let $\alpha\in\mathbb{R}$ and $\alpha\not\in\{0,1\}$.  Further denote $s_\alpha=\frac{\Delta^\alpha- \delta^\alpha}{\Delta- \delta}$. Then we have the following.
\begin{enumerate}
        \item If $\alpha<0$ or $\alpha>1$, the upper bound for the general first Zagreb index is $M_1^\alpha (G)\leq n\delta^\alpha +  (2m - n\delta)s_\alpha$. 
    
    \item If $0<\alpha<1$, it becomes the lower bound for the general first Zagreb index, that is, $M_1^\alpha (G)\geq n\delta^\alpha +  (2m - n\delta)s_\alpha$. 
\end{enumerate}
Moreover, for parts (1) and (2), the equality occurs for any bi-degreed graph with vertices of degree $\delta$ and $\Delta.$

\begin{enumerate}
\item[(3)] If $\alpha<0$ or $\alpha>1$, then the lower bound for the general first Zagreb index is $M_1^\alpha(G)\geq n\delta^\alpha +  (2m - n\delta) ((\delta + 1)^\alpha - \delta^\alpha) $, and the equality occurs for any graph with $\Delta=2$. Additionally, when $\Delta\geq 3$,
\begin{align*}
    M_1^\alpha(G)\geq n\delta^\alpha +  (2m - n\delta) (&(\delta + 1)^\alpha - \delta^\alpha) +
    \\&\Delta^\alpha - \delta^\alpha - (\Delta -\delta)((\delta + 1)^\alpha - \delta^\alpha) 
\end{align*}
    and the equality occurs if the graph G has a unique vertex of degree $\Delta$ and $n - 1$ vertices of degrees $\delta$ or $\delta + 1$ .
    
   \item[(4)] If $0<\alpha<1$, $n\delta^\alpha +  (2m - n\delta) ((\delta + 1)^\alpha - \delta^\alpha) $ becomes the upper bound for the general first Zagreb index, that is, $M_1^\alpha(G)\leq n\delta^\alpha +  (2m - n\delta) ((\delta + 1)^\alpha - \delta^\alpha) $. The equality occurs for any graph with $\Delta=2$. Additionally, when $\Delta\geq 3$, $M_1^\alpha(G)\leq n\delta^\alpha +  (2m - n\delta) ((\delta + 1)^\alpha - \delta^\alpha) +\Delta^\alpha - \delta^\alpha - (\Delta -\delta)((\delta + 1)^\alpha - \delta^\alpha)$, where the equality occurs if the graph G has a unique vertex of degree $\Delta$ and $n - 1$ vertices of degrees $\delta$ or $\delta + 1$ .
\end{enumerate}
       
\end{cor}
\begin{proof}
The results follow from Theorem \ref{UBFormulas}.
\end{proof}

\begin{example}
Here is an example. In Figure \ref{fig:1}, $\Delta=4, m=n=12$ and all vertices have either degree $1$ or $4$. It attains the bounds given in parts (1) and (2) of the Corollary \ref{Bounds_Cor}. For Figure \ref{fig:2}, $\Delta=4$, $n=m=6$. It  attains the bounds given in parts (3) and (4) of the Corollary \ref{Bounds_Cor}.
\begin{figure}[h]
\centering
\begin{minipage}{.4\textwidth}
  \centering
  \begin{tikzpicture}[scale=0.3]
\draw[ultra thick] (-3,3) -- (3,3) -- (3,-3)--(-3,-3) -- (-3,3);
\draw[ultra thick] (3,3)--(6,4);
\draw[ultra thick] (3,3)--(4,6);
\draw[ultra thick] (-3,3)--(-4,6);
\draw[ultra thick] (-3,3)--(-6,4);
\draw[ultra thick] (3,-3)--(6,-4);
\draw[ultra thick] (3,-3)--(4,-6);
\draw[ultra thick] (-3,-3)--(-6,-4);
\draw[ultra thick] (-3,-3)--(-4,-6);
\filldraw[black] (3,3) circle (10pt);
\filldraw[black] (-3,3) circle (10pt);
\filldraw[black] (3,-3) circle (10pt);
\filldraw[black] (-3,-3) circle (10pt);
\filldraw[black] (6,4) circle (10pt);
\filldraw[black] (4,6) circle (10pt);
\filldraw[black] (-6,4) circle (10pt);
\filldraw[black] (-4,6) circle (10pt);
\filldraw[black] (6,-4) circle (10pt);
\filldraw[black] (4,-6) circle (10pt);
\filldraw[black] (-6,-4) circle (10pt);
\filldraw[black] (-4,-6) circle (10pt);
\end{tikzpicture}
\caption{}
  \label{fig:1}
\end{minipage}%
\begin{minipage}{.4\textwidth}
  \centering
  \begin{tikzpicture}[scale=0.3]
\draw[ultra thick] (-3,3) -- (3,3) -- (3,-3)--(-3,-3) -- (-3,3);
\draw[ultra thick] (3,3)--(6,4);
\draw[ultra thick] (3,3)--(4,6);
\filldraw[black] (3,3) circle (10pt);
\filldraw[black] (-3,3) circle (10pt);
\filldraw[black] (3,-3) circle (10pt);
\filldraw[black] (-3,-3) circle (10pt);
\filldraw[black] (6,4) circle (10pt);
\filldraw[black] (4,6) circle (10pt);
\end{tikzpicture}
\caption{}
  \label{fig:2}
\end{minipage}
\end{figure}
\end{example}

In Theorem 2.2, Hu-Li-Shi-Xu-Gutman \cite{HLSXG} established formulas for the maximum and the minimum zeroth-order general Randic index (or, the  
  general first Zagreb index) for $(n, m)$ molecular graphs if $2m - n$ is congruent to 0, 1, or 2 modulo 3 and the maximum vertex degree $\Delta =  4 $. In the following Theorem, we will extend this and find sharp bounds for the same index if $2m - n\delta$ is congruent to any integer $r$ modulo $\Delta - \delta$ where $0 \le r < \Delta - \delta$ and the maximum vertex degree $\Delta \geq 3 $.

\begin{theorem}\label{UBFormulas_2}
Let $G$ be a graph with $n$ vertices and $m$ edges where $n\geq 3$. Let $n_i$ denote the number of vertices of degree $i$. Let $\delta$ and  $\Delta$ respectively denote the minimum and the maximum degree.  Assume that $\Delta\geq 3$, $\Delta - \delta > 1$ and $2m - n\delta = q(\Delta - \delta) + r$, where q is a positive integer and r is an integer such that $0 \leq r \leq \Delta - \delta - 1$. Let $\alpha\in\mathbb{R}$ and $\alpha\not\in\{0,1\}$.  Further denote $s_\alpha=\frac{\Delta^\alpha- \delta^\alpha}{\Delta- \delta}$. Then we have the following.    

 \begin{enumerate}
\item If $r = 0$ and $n_{\Delta} = q$, then the graph G is a bi-degreed graph with vertices of degree $\delta$ and $\Delta.$

        \item If $r \geq 1$ and $n_{\Delta} = q$, then $n_i = 0$ for $\delta + r + 1 \leq i \leq \Delta - 1$ and $n_{\delta + r} \leq 1$.

\item If $r \geq 1$ and $\alpha<0$ or $\alpha>1$ and $n_{\delta + r} \not = 0$ , then the upper bound for the general first Zagreb index is $M_1^\alpha (G)\leq n\delta^\alpha +  (2m - n\delta)s_\alpha + (\delta + r)^{\alpha} - \delta^\alpha - r s_{\alpha}$. 
    
   \item  If $r \geq 1$ and  $0<\alpha<1$ and $n_{\delta + r} \not = 0$, then it becomes the lower bound for the general first Zagreb index, that is, $M_1^\alpha (G)\geq n\delta^\alpha +  (2m - n\delta)s_\alpha + (\delta + r)^{\alpha} - \delta^\alpha - r s_{\alpha}$. 

 \end{enumerate}
        
Additionally, in both parts (3) and (4), the equality occurs if the graph G has $q$ vertices of degree $\Delta$, a unique vertex of degree $\delta + r $, and $n - q - 1$ vertices of degree $\delta$.
    
\end{theorem}

\begin{proof}
We have  $n=\sum_{i=\delta}^\Delta n_i$ and $\sum_{i=\delta}^\Delta in_i=2m$. By solving these equations we get $\sum_{i=\delta + 1}^\Delta (i - \delta)n_i=2m - n\delta = q(\Delta - \delta) + r$. If $r = 0$ and $n_{\Delta} = q$, then $n_i = 0$ for $ \delta + 1 \leq i \leq \Delta - 1$. So we get (1). Also if $r \geq 1$ and $n_{\Delta} = q$,then $n_i = 0$ for $\delta + r + 1 \leq i \leq \Delta - 1$ and $n_{\delta + r} \leq 1$. So we get (2). To prove (3), let $r\geq 1$ and $n_{\delta + r} \not =  0$. By part (1) of Theorem \ref{UBFormulas} we have the general first Zagreb index
\[M_1^{\alpha}(G)=n\delta^\alpha +  (2m - n\delta)s_\alpha +\sum_{i =1}^{\Delta - \delta - 1}n_{\delta + i} \left[(\delta + i)^\alpha - \delta^\alpha  - is_\alpha\right].\]

If $\alpha<0$ or $\alpha>1$, then the coefficients  $(\delta + i)^\alpha - \delta^\alpha - is_\alpha \leq 0$ for $1 \leq i \leq \Delta -\delta - 1$ . So 

\[M_1^{\alpha}(G)\leq n\delta^\alpha +  (2m - n\delta)s_\alpha + n_{\delta + r} \left[(\delta + r)^\alpha - \delta^\alpha  - rs_\alpha\right].\]

Since $n_{\delta + r } \geq 1$, we have
\[M_1^{\alpha}(G)\leq n\delta^\alpha +  (2m - n\delta)s_\alpha + \left[(\delta + r)^\alpha - \delta^\alpha  - rs_\alpha\right].\]

So the result in part (3) follows. Similarly, the result in part (4) follows since for $0<\alpha<1$, the coefficients $(\delta + i)^\alpha - \delta^\alpha - is_\alpha \geq 0$ for $1 \leq i \leq \Delta -\delta - 1$.
\end{proof}

\begin{example}
We note that for Figure \ref{fig:1}, we have  $\delta=1, \Delta=4, m=n=12$, $2m - n = 4(\Delta - 1), n_\Delta = 4$ and the the graph G is a bi-degreed graph with vertices of degree $1$ and $\Delta.$  

Let $ r$ be any positive integer. Consider a cycle with p vertices $V_1, V_2, \dots , V_p$ and attach $r - 1$ pendent vertices to the vertex $V_1$ and attach $r$ pendent vertices to each vertex $V_i$ for $2 \leq i \leq p$. Then we will have $\Delta=r + 2, m = n= p(r + 1) - 1$. So $2m - n = p(r + 1) - 1 = (p - 1)(r + 1) + r$ and  the graph has $q= p - 1$ vertices of degree $\Delta$, a unique vertex of degree $r + 1$, and $n - q - 1$ vertices of degree 1. It attains the bounds given in parts (3) and (4) of the Theorem \ref{UBFormulas_2}.

For the following Figure \ref{fig:3},  $\Delta=6, m=n=19$. So $2m - n = 19 = 3(\Delta - 1) + 4$ and the graph has $q=3$ vertices of degree $\Delta$, a unique vertex of degree $r + 1 = 5$, and $n - q - 1 = 15$ vertices of degree 1. It attains the bounds given in parts (3) and (4) of the Theorem \ref{UBFormulas_2}.
\begin{figure}[h]
\centering
\begin{tikzpicture}[scale=0.3]

\draw[ultra thick] (-3,3) -- (3,3) -- (3,-3)--(-3,-3) -- (-3,3);

\draw[ultra thick] (3,3)--(6,5);

\draw[ultra thick] (3,3)--(5,6);

\draw[ultra thick] (3,3)--(6.4,3.5);

\draw[ultra thick] (3,3)--(3.5,6.4);

\draw[ultra thick] (-3,3)--(-4,6);

\draw[ultra thick] (-3,3)--(-5.3,5.3);

\draw[ultra thick] (-3,3)--(-6,4);

\draw[ultra thick] (3,-3)--(6,-5);

\draw[ultra thick] (3,-3)--(5,-6);

\draw[ultra thick] (3,-3)--(6.4,-3.5);

\draw[ultra thick] (3,-3)--(3.5,-6.4);

\draw[ultra thick] (-3,-3)--(-6,-5);

\draw[ultra thick] (-3,-3)--(-5,-6);

\draw[ultra thick] (-3,-3)--(-6.4,-3.5);

\draw[ultra thick] (-3,-3)--(-3.5,-6.4);

\filldraw[black] (6.4,3.5) circle (10pt);

\filldraw[black] (6.4,-3.5) circle (10pt);

\filldraw[black] (3.5,-6.4) circle (10pt);

\filldraw[black] (3.5,6.4) circle (10pt);

\filldraw[black] (3,3) circle (10pt);

\filldraw[black] (-3,3) circle (10pt);

\filldraw[black] (3,-3) circle (10pt);

\filldraw[black] (-3,-3) circle (10pt);

\filldraw[black] (6,5) circle (10pt);

\filldraw[black] (5,6) circle (10pt);

\filldraw[black] (-6,4) circle (10pt);

\filldraw[black] (-4,6) circle (10pt);

\filldraw[black] (6,-5) circle (10pt);

\filldraw[black] (5,-6) circle (10pt);

\filldraw[black] (-6,-5) circle (10pt);

\filldraw[black] (-5,-6) circle (10pt);

\filldraw[black] (-6.4,-3.5) circle (10pt);

\filldraw[black] (-3.5,-6.4) circle (10pt);

\filldraw[black] (-5.3,5.3) circle (10pt);

\end{tikzpicture}
    \caption{}
    \label{fig:3}
\end{figure}

\end{example}

\section{Formulas and bounds for Leap Zagreb indices}

In this section we will establish some formulas for the general first leap Zagreb index for triangle and quadrangle free graphs.  We will then find sharp upper and lower bounds for the general first leap Zagreb index using these formulas. 

\begin{theorem}\label{UBFormulas3}
    Let $G$ be a triangle and quadrangle free graph with $n$ vertices and $m$ edges where $n\geq 3$. Let $n_i$ denote the number of vertices of $2$-distance degree $i$. Let $d$ and $D$ be respectively the minimum and the  maximum $2$-distance degree. Assume that $d\neq 0$, $D\not = d$, and $D\geq 2$.  Let $\alpha\in\mathbb{R}$ and $\alpha\not\in\{0,1\}$.  Further denote $S_\alpha=\frac{D^\alpha - d^\alpha }{D - d}$. Then we have the following.
    \begin{enumerate}
        \item The general first leap Zagreb index
        \[LM_1^\alpha(G)= nd^\alpha +  (M_1 - 2m - nd)S_\alpha +\sum_{i =1}^{D - d - 1}n_{d + i} \left[(d + i)^\alpha - d^\alpha - i S_\alpha\right].\]

        \item The general first leap Zagreb index
       \begin{align*}
           LM_1^\alpha (G) = nd^\alpha +  (M_1 &- 2m - nd)((d + 1)^\alpha - d^\alpha) +\\&\sum_{i =2}^{D - d }n_{d + i} \left[(d + i)^\alpha - d^\alpha - i ((d + 1)^\alpha - d^\alpha)\right].
       \end{align*}

        \item The first leap Zagreb coindex
            \[LM_1^{-}(G) = (M_1 - 2m) (n - 1 - D - d) + ndD - \sum_{i=1}^{D - d - 1} n_{d + i} ( i (d + i - D)).\]  
            
        \item The first leap Zagreb coindex
            \[LM_1^{-}(G) = (M_1 - 2m) (n - 2d - 2) + nd(d + 1) - \sum_{i=2}^{D - d} n_{d + i} i( i - 1).\]  
            
    \end{enumerate}
\end{theorem}
\begin{proof}
By Corollary 1.4 of Ferdose and Shivashankara \cite{FK}, we have $\sum_{i=d}^D in_i =  M_1 - 2m$. We also have $n=\sum_{i=d}^D n_i$. Solving these equations,   we will rewrite $n_D$ as follows.
\[
n_D=\frac{M_1 - 2m - nd - \sum_{i=d + 1}^{D  - 1}((i - d)n_i}{D - d}.
\]

Write $LM_1^\alpha (G)$ in terms of $n_i$, $d\leq i\leq D$, and substitute $n_d$ with $n-\sum_{i=d + 1}^D n_i$, we obtain
\[
LM_1^\alpha (G)=\sum_{i=d}^D  i^\alpha n_i = d^\alpha( n - \sum_{i=d + 1}^D n_i) + \sum_{i=d + 1}^D  i^\alpha n_i = nd^\alpha + \sum_{i = d + 1}^D  n_i(i^\alpha - d^\alpha).
\]
Further substitute $n_D$ into the sum, we get the result in part (1). Using the identities $n_d+n_{d + 1}=n - \sum_{i = d + 3}^D n_i$ and $dn_d + (d + 1)n_{d + 1} = M_1 - 2m - \sum_{i = d + 3}^D  in_i$, we solve for $n_d$ and $n_{d + 1}$ and substituting $n_d$ and $n_{d + 1}$ into $LM_1^\alpha(G)=\sum_{i=d}^D  i^\alpha n_i$, we get the result in part (2). To prove parts (3) and (4), we note that Ferdose and Shivashankara \cite{FK}  proved the following relation between the first leap Zagreb index and coindex for any $C_3, C_4$ free graph $G$,
\begin{equation}\label{eq}
     LM_1(G) + LM_1^{-}(G) =(n - 1)(M_1 - 2m).
\end{equation}
Then using parts (1) and (2), we get the results in parts (3) and (4).
\end{proof}

The following corollary gives upper bounds or lower bounds for the general first leap Zagreb index, depending on the values of $\alpha$. Both the upper bounds and lower bounds are sharp. 

\begin{cor}\label{cor:2}
    Let $G$ be a a triangle and quadrangle free graph with $n$ vertices and $m$ edges where $n\geq 3$. Let $n_i$ denote the number of vertices of 2-distance degree $i$. Let $d$ and $D$ be respectively the minimum and the  maximum 2-distance degree. Assume that $d\neq 0$, $D\not = d$, and $D\geq 2$.  Let $\alpha\in\mathbb{R}$ and $\alpha\not\in\{0,1\}$.  Further denote $S_\alpha=\frac{D^\alpha - d^\alpha }{D - d}$. Then we have the following. 

\begin{enumerate}
        \item If $\alpha<0$ or $\alpha>1$, the upper bound for the general first leap Zagreb index is $LM_1^\alpha (G)\leq  nd^\alpha +  (M_1 - 2m - nd)S_\alpha $. 
    
    \item If $0<\alpha<1$, it becomes the lower bound for the general first leap Zagreb index, that is, $LM_1^\alpha (G)\geq  nd^\alpha +  (M_1 - 2m - nd)S_\alpha$. 
\end{enumerate}
Moreover, for parts (1) and (2), the equality occurs for any bi-2-distance degreed graph with vertices of 2-distance degree $d$ and $D.$

\begin{enumerate}
\item[(3)] If $\alpha<0$ or $\alpha>1$, then the lower bound for the general first leap Zagreb index is $LM_1^\alpha(G)\geq  nd^\alpha +  (M_1 - 2m - nd)((d + 1)^\alpha - d^\alpha)$, and the equality occurs for any graph with $D = 2$. Additionally, when $D\geq 3$,
    \begin{align*}
    LM_1^\alpha(G)\geq nd^\alpha +  (2m - nD) (&(d + 1)^\alpha - d^\alpha) +
    \\&D^\alpha - d^\alpha - (D -d)((d + 1)^\alpha - d^\alpha) 
\end{align*}
    and the equality occurs if the graph G has a unique vertex of degree $D$ and $n - 1$ vertices of degree $d$ or $d+1$.
    
   \item[(4)] If $0<\alpha<1$, $nd^\alpha +  (M_1 - 2m - nd)((d + 1)^\alpha - d^\alpha)$ becomes the upper bound for the general first leap Zagreb index, that is, $LM_1^\alpha(G)\leq nd^\alpha +  (M_1 - 2m - nd)((d + 1)^\alpha - d^\alpha)$. The equality occurs for any graph with $D = 2$. Additionally, when $D\geq 3$, $LM_1^\alpha(G)\leq nd^\alpha +  (2m - nD) ((d + 1)^\alpha - d^\alpha) +
    D^\alpha - d^\alpha - (D -d)((d + 1)^\alpha - d^\alpha)$, where the equality occurs if the graph G has a unique vertex of 2-distance degree $D$ and $n - 1$ vertices of 2-distance degree $d$ or $d+1$. 
\end{enumerate}
       
\end{cor}

\begin{proof}
Follows from Theorem \ref{UBFormulas3} and Lemma \ref{Lemma}.
\end{proof}

\begin{example}
This is for the Corollary \ref{cor:2}. In the following Figure \ref{fig:4}, the graph is a bi-2-distance degree graph with vertices of 2-distance degree 1 and 4. It attains the bounds of parts (1) and (2) of the corollary. 
\begin{figure}[h]

\centering

  \begin{tikzpicture}[scale=0.4]

\draw[ultra thick] (-10,0)--(-5,0)--(0,0)--(5,0)--(10,0);

\draw[ultra thick] (0,0)--(-3,4);

\draw[ultra thick] (0,0)--(3,4)--(6,8);

\draw[ultra thick] (0,0)--(0,-5)--(0,-10);

\filldraw[black] (3,4) circle (10pt);

\filldraw[black] (6,8) circle (10pt);

\filldraw[black] (0,-5) circle (10pt);

\filldraw[black] (0,-10) circle (10pt);

\filldraw[black] (-10,0) circle (10pt);

\filldraw[black] (-5,0) circle (10pt);

\filldraw[black] (0,0) circle (10pt);

\filldraw[black] (5,0) circle (10pt);

\filldraw[black] (10,0) circle (10pt);

\filldraw[black] (-3,4) circle (10pt);

\end{tikzpicture}

\caption{}

  \label{fig:4}

\end{figure}

In the Figure \ref{fig:5}, the graph has a unique vertex of 2-distance degree 4 and all other vertices of 2-distance degree 1 or 2. It attains the bounds of parts (3) and (4) of the corollary.
\begin{figure}[h]

\centering

  \begin{tikzpicture}[scale=0.4]

\draw[ultra thick] (-12,0)--(-6,0)--(0,0)--(6,0)--(12,0);

\draw[ultra thick] (-6,0)--(-6,6);

\draw[ultra thick] (6,0)--(6,-6);

\filldraw[black] (-12,0) circle (10pt);

\filldraw[black] (-6,0) circle (10pt);

\filldraw[black] (0,0) circle (10pt);

\filldraw[black] (6,0) circle (10pt);

\filldraw[black] (12,0) circle (10pt);

\filldraw[black] (-6,6) circle (10pt);

\filldraw[black] (6,-6) circle (10pt);

\end{tikzpicture}

 \caption{}

  \label{fig:5}

\end{figure}
\end{example}

In the following Theorem, we will find sharp bounds for the general first leap Zagreb  index if $M_1 - 2m - nd = q(D - d) + r$, where $q$ is a positive integer and r is an integer such that $0 \leq r \leq D - d - 1$.  

\begin{theorem}\label{UBFormulas_3}
   Let $G$ be a triangle and quadrangle free graph with $n$ vertices and $m$ edges where $n\geq 3$. Let $n_i$ denote the number of vertices of 2-distance degree $i$. Let $d$ and $D$ be respectively the minimum and the  maximum 2-distance degree. Assume that $d\neq 0$, $D - d\geq 2$, and $M_1 - 2m - nd = q(D - d) + r$, where $q$ is a positive integer and r is an integer such that $0 \leq r \leq D - d - 1$.   Let $\alpha\in\mathbb{R}$ and $\alpha\not\in\{0,1\}$.  Further denote $S_\alpha=\frac{D^\alpha - d^\alpha }{D - d}$. Then we have the following.
    
 \begin{enumerate}
        \item If $r = 0$ and $n_{D} = q$, then the graph G is a bi-2-distance degree graph with vertices of 2-distance degree d and $D.$. 

  \item If $r \geq 1$ and $n_{D} = q$, then $n_i = 0$ for $ d + r + 1 \leq i \leq D  - 1$ and $n_{d + r} \leq 1$.

\item If $r \geq 1$ and $\alpha<0$ or $\alpha>1$ and $n_{d + r} \not = 0$ , then the upper bound for the general first leap Zagreb index is
 $LM_1^{\alpha}(G)\leq nd^\alpha +  (M_1 - 2m - nd)S_\alpha + ((d + r )^\alpha - d^\alpha - rS_\alpha) $.

   \item  If $r \geq 1$ and $0<\alpha<1$ and $n_{d + r} \not = 0$, it becomes the lower bound for the general first leap Zagreb index, that is, $LM_1^{\alpha}(G)\geq nd^\alpha +  (M_1 - 2m - nd)S_\alpha + ((d + r )^\alpha - d^\alpha - rS_\alpha) $.  

 \end{enumerate}
        
Additionally, in both parts (3) and (4) the equality occurs if the graph G has $q$ vertices of 2-distance degree $D$, a unique vertex of 2-distance degree $r + 1$, and $n - q - 1$ vertices of 2-distance degree $d$.
 
\end{theorem}

\begin{proof}
We have  $n=\sum_{i=d}^D  n_i$ and $\sum_{i = d}^D  in_i= M_1-  2m$. By solving the equations, we get $\sum_{i=d + 1}^D  (i - d)n_i= M_1 - 2m - nd = q(D - d) + r$. If $r = 0$ and $n_{D} = q$, then $n_i = 0$ for $ d + 1 \leq i \leq D - 1$. So we get the result in (1). Also if $r \geq 1$ and $n_{D} = q$,then $n_i = 0$ for $ d + r + 1 \leq i \leq D - 1$ and $n_{d + r} \leq 1$. So we get the result in (2). To prove (3), let $r\geq 1$ and $n_{d + r} \not =  0$. By part (1) of Theorem \ref{UBFormulas3} we have the general first leap Zagreb index 
$LM_1^\alpha(G)= nd^\alpha +  (M_1 - 2m - nd)S_\alpha +\sum_{i =1}^{D - d - 1}n_{d + i} \left[(d + i)^\alpha - d^\alpha - i S_\alpha\right]$. By part (1) of lemma (3.1), it follows that  if $\alpha<0$ or $\alpha>1$, then $(d + i)^\alpha - d^\alpha - i S_\alpha  \leq 0$ for $1 \leq i \leq D - d - 1$. So $LM_1^\alpha(G) \leq nd^\alpha +  (M_1 - 2m - nd)S_\alpha +  n_{d + r} \left[(d + r)^\alpha - d^\alpha - rS_\alpha\right]$. Since $n_{r + d} \not =  0$, the result in part (3) follows. Similarly, the result in part (4) follows since for $0<\alpha<1$, $(d + i)^\alpha - d^\alpha - i S_\alpha  \geq 0$ for $1 \leq i \leq D - d - 1$.

\end{proof}

\begin{example}
Let $a$ be any positive integer greater than 1. Consider the star graph with p vertices $V_1, V_2, \dots, V_p$, where $p > 4$, $p - 2 > a$ and the vertex $V_1$ is the center. Attach one pendent vertex to each vertex $V_i$ for $2 \leq i \leq a$. Then we will have number of vertices $ n = p + a$, the number of edges  $m = p + a - 1$, the 2-distance degree of the center vertex $V_1$ is $a$, the 2-distance degree of each vertex $V_i = p -2$ for $2 \leq i \leq p$, and the 2-distance degree of each pendent vertex is 1. So the maximum 2-distance degree $D = p - 2$ and the minimum 2-distance degree $d = 1$. So, $M_1 - 2m - nd = p^2 - 4p + a + 2 = (p -1)(p - 3) + a - 1$, and  the graph has $q= p - 1$ vertices of 2-distance degree $D$, a unique vertex of 2-distance degree $a$, and $n - q - 1$ vertices of 2-distance degree 1. It attains the bounds given in parts (3) and (4) of the Theorem \ref{UBFormulas_3}. 

For the following Figure \ref{fig:6},  $D = 4, d = 1, p = 6, a = 3, n=9$. So $M_1 - 2m - nd = 17 = 5(D - 1) + 2$ and the graph has $q=5$ vertices of 2-distance degree $D$, a unique vertex of 2-distance degree $3$, and $n - q - 1 = 3$ vertices of 2-distance degree 1. It attains the bounds given in parts (3) and (4) of the Theorem \ref{UBFormulas_3}.

\begin{figure}[h]

\centering

  \begin{tikzpicture}[scale=0.4]

\draw[ultra thick] (-5,0)--(0,0)--(5,0)--(10,0);

\draw[ultra thick] (0,0)--(-3,4);

\draw[ultra thick] (0,0)--(3,4)--(6,8);

\draw[ultra thick] (0,0)--(0,-5)--(0,-10);

\filldraw[black] (3,4) circle (10pt);

\filldraw[black] (6,8) circle (10pt);

\filldraw[black] (0,-5) circle (10pt);

\filldraw[black] (0,-10) circle (10pt);

\filldraw[black] (-5,0) circle (10pt);

\filldraw[black] (0,0) circle (10pt);

\filldraw[black] (5,0) circle (10pt);

\filldraw[black] (10,0) circle (10pt);

\filldraw[black] (-3,4) circle (10pt);

\end{tikzpicture}

\caption{}

  \label{fig:6}

\end{figure}

\end{example}

\begin{remark}
    If $d = 0$, we have the following formula.

The general first Leap Zagreb index
        \[LM_1^\alpha(G)= (n - n_0) +  (M_1 - 2m - n + n_0) (\frac{D^\alpha - 1 }{D - 1}) +\sum_{i =2}^{D - - 1}n_{i} \left[(i)^\alpha - 1 - (i - 1)(\frac{D^\alpha - 1 }{D - 1})\right].\]
As said in the proof of Theorem (4.1) we have $\sum_{i=d}^D in_i =  M_1 - 2m$ and $n=\sum_{i=d}^D n_i$. Solving these equations for $n_1$ and $n_D$, we get the formula. In addition, solving the equations for $n_1$ and $n_0$, we get the following formula.The general first Leap Zagreb index
$ LM_1^\alpha (G) = M_1 - 2m +\sum_{i =2}^{D}n_i (i^\alpha - i)$.
\end{remark}

\section{Regression models for Chemical Compounds}
In this Section, we will find regression models for properties like boiling point and entropy of Benzenoid hydrocarbons using the first Zagreb index and the first leap Zagreb index.

In 2022, P.Sarkar, Nilanjan De, and Anita Pal \cite{SPNA} investigated predictive ability of physicochemical prperties of plycyclic aromatic hydrocarbons for cerain degree based topological indices including Zagreb indices. They found the following:
The correlation coefficient of the first Zagreb index $M_1$ and the entropy S of the data of 22 benzenoid hydrocarbons given in Table 2 of the paper = 0.923654 and the regression equation is 
$S = 62.277(+ \text{ or} -4.626) + 0.416(+ \text{ or} - 0.038)M_1$.

We found the following: 
The correlation coefficient of the first leap zagreb index $LM_1$ and the entropy S of the data of 22 benzenoid hydrocarbons given in Table 2 of the paper = 0.867332963 and the regression equation is 
$S = 76.29013499  + 0.127338078LM_1$.

In addition, In 2023, H.S. Ramane and K.S. Pise \cite{PKH} showed chemical importance of the leap Zagreb indices. They found the following:
The correlation coefficient of the first leap Zagreb index $LM_1$ and the boiling points of the data of 21 benzenoid hydrocarbons given in Fig.1 of the paper = 0.9656 and the regression equation is 
$BP= 171.6991  + 1.1451LM_1$.

We found the following:
The correlation coefficient of the first zagreb index $M_1$ and the boiling points of the data of 21 benzenoid hydrocarbons given in Fig.1 of the paper = 0.992773893 and the regression equation is 
$BP= 58.08410846 +  3.628596366M_1$.

 Overall, Table \ref{tab:1} is a summary  of the regression models. For both properties, the first Zagreb index correlates better compared to the first leap Zagreb index.

\begin{table}[h]
    \centering
    \resizebox{\columnwidth}{!}{
    \begin{tabular}{|c|c|c|c|c|}\hline
    & $\rho$ ($M_1$) & Regression ($M_1$) & $\rho$ ($LM_1$) & Regression ($LM_1$)\\\hline
  Correlation with $S$      & $0.923654$  & $S = 62.277 + 0.416M_1$ & $0.867332963$ &$S = 76.290  + 0.127LM_1$\\\hline
  Correlation with $BP$       & $0.992773893$ & $BP= 58.084 +  3.629M_1$& $0.9656$ & $BP= 171.699  + 1.145LM_1$\\\hline
    \end{tabular}}
    \caption{}
    \label{tab:1}
\end{table}

\section{Discussion and Conclusion}
In the last forty years, many scientists have developed mathematical models for analyzing structures and properties of various chemical compounds. Graph theory has provided many powerful tools to develop many types of models such as Quantitative Structure-Property Relationships (QSPR) models and Quantitative Structure-Activity Relationships (QSAR) models. The pioneering work of Harry Weiner inspired many mathematicians and scientists to define more topological indices and study their applications to various fields such as medicine and environment.

The main research question is how to compute bounds and formulas of the various topological indices and find graphs which attain the bounds. In this paper, we established formulas of the general first Zagreb index, coindex, and leap Zagreb indices of graphs. Additionally, we found upper and lower bounds for the general first Zagreb index and leap Zagreb index and determine graphs which attain the bounds. In addition, we established sharp upper and lower bounds for the said indices for some special graphs.

Since there are more than 100 topological indices, there are many open problems of finding bounds and formulas for them. They play a vital role in modeling various properties such as physico-chemical, and biological properties of chemical compounds. This is silmply amazing!

\bibliographystyle{amsplain}

\end{document}